\documentclass{amsart}
\usepackage{graphicx}
\usepackage[mathscr]{eucal}
\usepackage{amssymb}
\newtheorem{thm}{Theorem}[section]
\newtheorem{cor}[thm]{Corollary}
\newtheorem{lem}[thm]{Lemma}
\newtheorem{prop}[thm]{Proposition}
\theoremstyle{definition}

\theoremstyle{remark}
\newtheorem{rem}[thm]{Remark}

\newcommand{\set}[1]{\left\{#1\right\}}

\newcommand{\A}{\mathcal{A}}

\newcommand{\E}{\mathcal{E}}
\newcommand{\Li}{\mathcal{L}}
\newcommand{\ML}{\mathcal{M}\mathcal{L}}
\newcommand{\F}{\mathcal{F}}
\newcommand{\pr}{{}^{\prime}}
\newcommand{\Fp}{\F\pr}
\newcommand{\Lp}{\Li\pr}
\newcommand{\mset}{\emptyset}
\newcommand{\U}{\mathcal{U}}
\newcommand{\V}{\mathcal{V}}
\begin{document}
\baselineskip=18pt
\title{Hyperspaces of closed limit sets}
\author{Aldo J. Lazar}
\address{School of Mathematical Sciences\\
         Tel Aviv University\\
         Tel Aviv 69778, Israel}
\email{aldo@post.tau.ac.il}

\thanks{}%

\date{}%
\begin{abstract}
   We study Michael's lower semifinite topology and Fell's topology on the collection of all closed limit subsets of
   a topological space. Special attention is given to the subfamily of all maximal limit sets.
\end{abstract}
\maketitle
\section{Introduction} \label{S:Intro}

The collection of all closed subsets of a topological space has been for long of interest to topologists and
functional analysts. It seems that the modern investigation of the subject began with \cite{M}. It is well known
that there is a one-to-one correspondence between the closed two-sided ideals of a $C^*$-algebra and the closed
subsets of its primitive ideal space as detailed in \cite[Proposition 3.2.2]{DC}. Naturally, this correspondence
attracted the interest of operator algebraists in the hyperspace of the closed subsets of a topological space. It
led Fell to the definition in \cite{F} of a topology on this hyperspace that is of significance in topology and
several branches of analysis. Moreover, according to \cite[Proposition 3.2]{AB}, when one restricts this
correspondence to the closed limit subsets of the primitive ideal space, a very interesting class of ideals is
obtained. The wealth of information given in \cite{A} on this class of ideals stimulated the present investigation
and a significant portion of the results that appear here were proved in \cite{A} for this special family of ideals
of a $C^*$-algebra. However, no knowledge of the theory of $C^*$-algebras is required for the understanding of the
following; we discuss the properties of two topologies on the collection of all the closed limit subsets of a
topological space. All the definitions beyond the common knowledge of a topologist or an analyst are given in the
next section. Of course, all our results are significant only for non Hausdorff spaces, as the primitive ideal
spaces often are.

In section \ref{S:weak} we study the Michael's lower semifinite topology on the family of all closed limit sets. We
establish that with this topology this hyperspace is a locally compact Baire space. We restrict the discussion to
the collection of all maximal limit sets in section \ref{S:maximal}. The Fell topology and the lower semifinite
topology coincide on this hyperspace. This hyperspace is also a Baire space and if the initial space is second
countable and locally compact then the hyperspace of maximal limit sets is a $G_{\delta}$ subspace in the space of
all closed limit sets equipped with the Fell topology.

\section{Preliminaries} \label{S:Pre}

For a topological space $X$ we shall denote by $\F(X)$ the hyperspace of all its closed subsets and $\Fp(X)$ will
stand for the collection of all the non-void closed subsets of $X$. A subset $L$ of $X$ is called a limit set if
there is a net that converges to all the points of $L$. By \cite[Lemme 9]{D}, $L\subset X$ is a limit set if and
only if every finite family of open subsets that intersect $L$ has a non-void intersection. The collection of all
the closed limit sets of $X$ will be denoted by $\Li(X)$ and we set $\Lp(X) := \Li(X)\cap \Fp(X)$. It easily follows
from the lemma quoted above and Zorn's lemma that each $L\in \Li(X)$ is contained in a maximal limit set. Obviously,
every maximal limit set is closed and non-void. $\ML(X)$ will denote the collection of all maximal limit sets. There
is a natural map $\eta_X : X\to \Lp(X)$ defined by $\eta_X(x) := \overline{\set{x}}$. This map is one to one if and
only if $X$ is a $T_0$ space.

Some of the results below are valid under the restriction that the topological space $X$ is locally compact that is,
each point in $X$ has a fundamental system of compact neighbourhoods. Such spaces were called locally quasi-compact
in \cite[I, 9, Ex. 29]{B}.

For $C$ be a compact subset and $\Phi$ a finite family of open subsets of $X$ let
\[
 \U(C,\Phi) := \set{A\in \F(X) \mid A\cap C = \mset, A\cap O \neq \mset, O\in \Phi}.
\]
The collection of all such $\U(C,\Phi)$ forms a base for a topology on $\F(X)$ that was defined by Fell in \cite{F}
and which will be denoted here by $\tau_s$. It was shown in \cite{F} that with this topology $\F(X)$ is a compact
space that is Hausdorff if $X$ is locally compact. If $X$ is locally compact and has a countable base then
$(\F(X),\tau_s)$ is metrizable, see \cite[Lemme 2]{D}.

The collection of all $\U(\mset,\Phi)$ when $\Phi$ runs through all the finite families of open subsets of $X$ is
the base of a $T_0$ topology on $\F(X)$, weaker than $\tau_s$, which we shall denote by $\tau_w$. It was called the
lower semifinite topology in \cite[Definition 9.1]{M} and was further discussed in \cite{K}. It is easily seen that
if $\mathcal{B}$ is a base for the topology of $X$ then the collection of all $\U(\mset,\Phi)$ when $\Phi$ runs
through all the finite subfamilies of $\mathcal{B}$ is a base for $(\F(X),\tau_w)$. Thus, if $X$ is second countable
then $(\F(X),\tau_w)$ is also second countable. Clearly $\Fp(X) = \U(\mset,\set{X})$ hence $\Fp(X)$ is $\tau_w$-open
in $\F(X)$. The only $\tau_w$-open subset of $\F(X)$ to which the empty subset of $X$ belongs is $\F(X)$ itself so
$\Fp(X)$ is $\tau_w$-dense in $\F(X)$ and $\Lp(X)$ is $\tau_w$-dense in $\Li(X)$. Obviously, $\ML(X)$ is also
$\tau_w$-dense in $\Li(X)$. For every $A\in \F(X)$ the $\tau_w$-closure of $\set{A}$ is $\set{B\in \F(X) \mid
B\subset A}$ and this entails the $T_0$ separation property for$(\F(X),\tau_w)$ . The map $\eta_X$ is $\tau_w$
continuous; it is a homeomorphism on its image if $X$ is $T_0$. Generalizing \cite[Proposition 3.1]{A}, we claim
that always the $\tau_w$-closure of $\eta_X(X)$ is $\Li(X)$. Indeed, it is easily seen that $A\in \F(X)$ is in the
$\tau_w$-closure of $\eta_X(X)$ if and only if every finite family of open subsets that intersect $A$ has a non-void
intersection that is, if and only if $A\in \Li(X)$. In particular, $\Li(X)$ is $\tau_w$-closed hence also
$\tau_s$-closed. Thus $(\Li(X),\tau_s)$ is a compact Hausdorff space. From the $\tau_w$-density of $\eta_X(X)$ in
$\Li(X)$ it follows that $(\Li(X),\tau_w)$ is connected when $X$ is connected. However, trivial examples show that
$(\Li(X),\tau_s)$ need not be connected if $X$ is connected.

Concerning the $\tau_s$-convergence of nets the following was proved in \cite[Lemma H.2]{W}:

\begin{prop} \label{P:conv1}

   Let $\set{A_{\iota}}$ be a net of closed subsets of the topological space $X$ and $A\in \F(X)$. The net
   $\tau_s$-converges to $A$ if (\textrm{a}) given $x_{\iota}\in A_{\iota}$ such that the net $\set{x_{\iota}}$
   converges to $x$, then $x\in A$, and (\textrm{b}) if $x\in A$ then there is a subnet $\set{A_{\iota_{\kappa}}}$
   and points $x_{\iota_{\kappa}}\in A_{\iota_{\kappa}}$ such that $\set{x_{\iota_{\kappa}}}$ converges to $x$. When
   $X$ is locally compact the converse is true too : the net $\set{A_{\iota}}$ $\tau_s$-converges to $A$ only if the
   conditions (\textrm{a}) and (\textrm{b}) hold.

\end{prop}

The characterization of the $\tau_s$-convergence of nets given below is in line with our attempt to investigate the
links between the two topologies on the hyperspace of closed subsets noted above. A net in a topological space was
called by Fell primitive in \cite{F} if the set of all its limits equals the set of all its cluster points. With
this definition we have

\begin{prop} \label{P:conv2}

   Let $X$ be a topological space. If $\set{A_{\iota}}$ is a primitive net in $(\F(X),\tau_w)$ and the set of
   all its $\tau_w$-limits is $\set{B\in \F(X) \mid B\subset A}$ where $A\in \F(X)$ then $\set{A_{\iota}}$
   $\tau_s$-converges to $A$. If $X$ is locally compact then the converse holds: a net $\set{A_{\iota}}$ that is
   $\tau_s$-convergent to $A$ in $\F(X)$ is primitive in $(\F(X),\tau_w)$ and the set of all its $\tau_w$-limits is
   $\set{B\in \F(X) \mid B\subset A}$.

\end{prop}

\begin{proof}

   Suppose $\set{A_{\iota}}$ is a $\tau_w$-primitive net in $\F(X)$ and the set of all its limits is $\set{B \mid
   B\subset A}$. Let $\U(C,\Phi)$ be a basic $\tau_s$-neighbourhood of $A$. If we assume that $\set{A_{\iota}}$ is
   not eventually in $\U(C,\Phi)$ then, by passing to a subnet and relabelling, we have $A_{\iota}\cap C\neq \mset$
   for each $\iota$. We choose points $x_{\iota}\in A_{\iota}\cap C$. There is a subnet $\set{x_{\iota_{\kappa}}}$
   that converges to a point $x$ in the compact set $C$. We claim that $\set{A_{\iota_{\kappa}}}$ $\tau_w$-converges
   to $\overline{\{x\}}$. Indeed, let $\Phi_1$ be a finite family of open subsets of $X$ all of which intersect
   $\overline{\{x\}}$ that is, such that $x$ belongs to the intersection $V$ of all the sets in $\Phi_1$. Then
   $x_{\iota_{\kappa}}$ is eventually in $V$. Thus, for $\kappa$ large enough $A_{\iota_{\kappa}}\cap V\neq \mset$
   and the claim is established. Thus $\overline{\{x\}}$ is a $\tau_w$-cluster point of the primitive net
   $\set{A_{\iota}}$ hence $\overline{\{x\}}\subset A$. We got $A\cap C\neq \mset$, a contradiction.

   Suppose now that $X$ is locally compact and the net $\{A_{\iota}\}$ $\tau_s$-converges to $A$. It follows readily
   from the definition of the topologies on $\F(X)$ that $\{A_{\iota}\}$ $\tau_w$-converges to every closed subset $B$ of
   $X$ which is a subset of $A$. Assume that there is a subnet $\{A_{\iota_{\kappa}}\}$ that $\tau_w$-converges to
   some $B\in \F(X)$ with $B\setminus A\neq \mset$ and let $x\in B\setminus A$. There is a compact set $C\subset X$
   such that $x\in Int(C)\subset C\subset X\setminus A$. $\U(C,\{X\})$ is a $\tau_s$-neighbourhood of $A$ hence
   $A_{\iota_{\kappa}}\cap C = \mset$ eventually. On the other hand, $\U(\mset, \{Int(C)\})$ is a
   $\tau_w$-neighbourhood of $B$ hence $A_{\iota_{\kappa}}\cap Int(C)\neq \mset$ and we got a contradiction. We have
   proved that each $\tau_w$-cluster point of $\{A_{\iota}\}$ is a subset of $A$ and we are done.

\end{proof}

\section{the topology $\tau_w$} \label{S:weak}

First we want to establish the local compactness of $\F(X)$, $\Fp(X)$, $\Li(X)$, and $\Lp(X)$ with their
$\tau_w$-topology when the space $X$ is locally compact. The result for the first two spaces is likely to be known
but we have no reference for it. The local compactness of $\Li(X)$ and $\Lp(X)$ was established when $X$ is the
primitive ideal space of a $C^*$-algebra in \cite[Theorem 3.7]{A} by using special properties of such spaces.

\begin{lem} \label{L:compactness}

   Let $C_1, \ldots C_n$ be compact subsets of the topological space $X$. Then $\mathcal{S} := \set{A\in \F(X) \mid
   A\cap C_i \neq \emptyset, 1\leq i\leq n}$ is $\tau_w$-compact.

\end{lem}

\begin{proof}

   Let $\set{M_{\alpha} \mid \alpha\in \mathcal{A}}$ be a net in $\mathcal{S}$ and $x_{\alpha}^i\in M_{\alpha}\cap C_i$.
   By passing to successive subnets we may suppose that each of the nets $\set{x_{\alpha}^i \mid \alpha\in
   \mathcal{A}}$, $1\leq i\leq n$, converges to a point $x_i\in C_i$. Denote by $M$ the closure of $\set{x_1, \dots
   x_n}$ and suppose $\E := \set{U_1, \ldots U_p}$ is a finite family of open subsets of $X$ such that
   $M\in \mathscr{U}(\mset, \E)$. Then for each $k$, $1\leq k\leq p$, there is $1\leq i_k\leq n$ such that
   $x_{i_k}\in U_k$. Hence there is $\alpha_0\in \A$ such that for all $1\leq k\leq p$ and $\alpha > \alpha_0$ we
   have $x_{\alpha}^{i_k}\in U_k$. Thus, if $\alpha > \alpha_0$ then $M_{\alpha}\cap U_k\neq \mset$, $1\leq k\leq
   p$. We have established that $\set{M_{\alpha}}$ converges weakly to $M$ and clearly $M\in \mathcal{S}$.

\end{proof}

\begin{thm} \label{T:lc}

   If $X$ is a locally compact space then $\F(X)$, $\Fp(X)$, $\Li(X)$, and $\Lp(X)$ are locally compact spaces with
   their $\tau_w$ topology.

\end{thm}

\begin{proof}

   Suppose $X$ is a locally compact space and let $A$ be a closed subset of $X$. For a basic $\tau_w$-neighbourhood
   $\U(\mset,\set{U_i}_{i=1}^n)$ of $A$ we choose $x_i\in A\cap U_i$, $1\leq i\leq n$. Let $V_i$ be a compact
   neighbourhood of $x_i$ contained in $U_i$ and $W_i := Int(V_i)$. Then
   $$
    A\in \U(\mset,\set{W_i}_{i=1}^n)\subset \mathcal{V} := \set{B\in \F(X) \mid B\cap V_i\neq \mset, 1\leq i\leq n}\subset
                                                                                                \U(\mset,\{U_i\}_{i=1}^n).
   $$
   Thus $\mathcal{V}$ is a neighbourhood of $A$ that is compact by the preceding lemma. We have proved that
   $(\F(X),\tau_w)$ is locally compact.

   As remarked above, $\Fp(X)$ is $\tau_s$-open in $F(X)$, $\Li(X)$ is $\tau_w$-closed, $\Lp(X) = \Li(X)\cap
   \Fp(X)$ is relatively open in $\Li(X)$ and the conclusion follows.

\end{proof}

The next result was stated in \cite[Proposition 3.4]{A} for the primitive ideal space of a $C^*$-algebra. However,
the proof given there is valid for any topological space and we reproduce it here.

\begin{prop} \label{P:Baire}

   If $X$ is a Baire topological space then $(\Li(X),\tau_w)$ and $(\Lp(X),\tau_w)$ are Baire spaces.

\end{prop}

\begin{proof}

   For each natural number $n$ let $\U_n$ be a $\tau_w$-dense open subset of $\Li(X)$. Since $\eta_X(X)$ is
   $\tau_w$-dense in $\Li(X)$ and $\eta_X$ is $\tau_w$-continuous, $\eta_X^{-1}(\U_n)$ is an open dense subset of
   $X$. From the hypothesis it follows that $\cap_{n\geq 1} \eta_X^{-1}(\U_n)$ is dense in $X$. But then
   $$
    \eta_X(\cap_{n\geq 1} \eta_X^{-1}(\U_n)) = \eta_X(X)\bigcap (\cap_{n\geq 1} \U_n)
   $$
   is $\tau_w$-dense in $Li(X)$. In particular, $\cap_{n\geq 1} \U_n$ is $\tau_w$-dense in $\Li(X)$.

   $\Lp(X)$ is an open dense subset of $(\Li(X),\tau_w)$ so it is a Baire space too.

\end{proof}

\begin{prop} \label{P:opencompact}

   If $X$ has a base consisting of open and compact sets then the same is true for $(\F(X),\tau_w)$ and its subspaces
   $\Fp(X)$, $\Li(X)$, and $\Lp(X)$.

\end{prop}

\begin{proof}

   Suppose $\mathcal{B}$ is a base for the topology of $X$ consisting of open and compact sets. Then the collection
   of all the families $\U(\mset,\Phi)$ where $\Phi$ runs through all the finite subfamilies of $\mathcal{B}$
   is a base for $(\F(X),\tau_w)$. Each $\U(\mset,\Phi)$ is $\tau_w$-compact by Lemma \ref{L:compactness}. We get a base
   for $\Fp(X)$ by requiring $\Phi$ to run through the nonempty finite subfamilies of $\mathcal{B}$. Intersecting
   each of the elements of the bases we got for $\F(X)$ and $\Fp(X)$ with the $\tau_w$-closed set $\Li(X)$ we get bases as
   needed for $\Li(X)$ and $\Lp(X)$, respectively.

\end{proof}

\section{The hyperspace $\ML(X)$} \label{S:maximal}

The next result generalizes \cite[Theorem 4.2]{A} where the framework is that of a certain family of ideals of a
$C^*$-algebra and the proof uses $C^*$-algebraic methods. The "if" part of the statement is also a consequence of
\cite[Lemme 15]{D}.

\begin{thm} \label{T:cont}

   The identity map $(\Li(X),\tau_w)\to (\Li(X),\tau_s)$ is continuous at $A\in \Li(X)$ if and only if $A\in
   \ML(X)$.

\end{thm}

\begin{proof}

Suppose $A$ is a maximal limit set. Let $C$ be a compact subset of $X$ and $\Phi$ a finite family of open subsets of
$X$ such that $A\in \U(C,\Phi)$. We claim that there is a finite family $\Psi\supset \Phi$ of open subsets of $X$
each of which has a nonempty intersection with $A$ and such that $\U(\mset,\Psi)\cap \Li(X)\subset \U(C,\Phi)\cap
\Li(X)$. This, of course, will establish the continuity of the identity map at $A$.

Assume there is no such $\Psi$. Then for each finite family $\Psi\supset \Phi$ of open subsets of $X$ such that
every set in $\Psi$ has a nonempty intersection with $A$ there is $B_{\Psi}\in (\U(\mset,\Psi)\setminus
\U(C,\Phi))\cap \Li(X)$. Denote the collection of all such families $\Psi$ by $\boldsymbol{\Lambda}$ and order it by
inclusion. Clearly $\Psi\in \boldsymbol{\Lambda}$ implies $B_{\Psi}\cap C\neq \mset$. Choose $x_{\Psi}\in
B_{\Psi}\cap C$. The net $\set{x_{\Psi}}$ has a converging subnet to some point $x\in C$. We have $x\notin A$ hence
$A\cup \{x\}\supsetneqq A$. We shall show that $A\cup \{x\}$ is a limit set hence $A\cup \overline{\{x\}}\in
\Lp(X)$, and this will yield a contradiction to the maximality of $A$.

Let $\mathcal{N}$ be the family of all the open neighbourhoods of $x$. We order $\mathcal{N}\times
\boldsymbol{\Lambda}$ by defining $(V_1,\Psi_1)\prec (V_2,\Psi_2)$ if $V_1\supset V_2$ and $\Psi_1\subset \Psi_2)$.
Denote by $\boldsymbol{\Gamma}$ the collection of all the pairs $(V,\Psi)\in \mathcal{N}\times \boldsymbol{\Lambda}$
such that the finite family of open sets $\{V\}\cup \Psi$ has a nonempty intersection. For $(V_1,\Psi_1),
(V_2,\Psi_2)\in \mathcal{N}\times \boldsymbol{\Lambda}$ there is $(V,\Psi)\in \boldsymbol{\Gamma}$ such that
$(V_1,\Psi_1)\prec (V,\Psi)$ and $(V_1,\Psi_2)\prec (V,\Psi)$. Indeed, $V := V_1\cap V_2$ is an open neighbourhood
of $x$ and $\Psi_1\cup \Psi_2\in \boldsymbol{\Lambda}$ hence there is $\Psi\in \boldsymbol{\Lambda}$ that satisfies
$\Psi\supset \Psi_1\cup \Psi_2$ and $x_{\Psi}\in V$. Thus $x_{\Psi}\in V\cap B_{\Psi}$ and since $B_{\Psi}$ is a
limit set that belongs to $\U(\mset,\Psi)$, the family of open sets $\{V\}\cup \Psi$ has a nonempty intersection by
the previously quoted Lemme 9 of \cite{D}. We got $(V,\Psi)\in \boldsymbol{\Gamma}$ as needed. In particular,
$\boldsymbol{\Gamma}$ is a directed set with this order restricted to it. For each $(V,\Psi)\in \boldsymbol{\Gamma}$
we choose $y_{(V,\Psi)}$ in the intersection of the family $\{V\}\cup \Psi$. The net $\{y_{(V,\Psi)}\}$ converges to
every point of $\{x\}\cup A$. It is clear that the net converges to $x$. Let now $y$ be a point of $A$ and $W$ an
open neighbourhood of $y$. With $\Psi_0 := \{W\}\cup \Phi$ we have $(X,\Psi_0)\in \mathcal{N}\times
\boldsymbol{\Lambda}$. By the order property of $\boldsymbol{\Gamma}$ proved above there is $(V_1,\Psi_1)\in
\boldsymbol{\Gamma}$ such that $(X,\Psi_0)\prec (V_1,\Psi_1)$. Clearly if $(V,\Psi)\in \boldsymbol{\Gamma}$ and
$(V_1,\Psi_1)\prec (V,\Psi)$ then $W\in \Psi$ hence $y_{(V,\Psi)}\in \cap \{O \mid O\in \Psi\}\subset W$. We have
proved that the net $\{y_{(V,\Psi)} \mid (V,\Psi)\in \boldsymbol{\Gamma}\}$ converges to $y$ as claimed.

Let now $L$ be a non-maximal closed limit set of $X$. There are $z\in X\setminus L$ and a net that converges to all
the points of $L\cup \overline{\{z\}}$. The set $L$ belongs to the $\tau_s$-open set $\U(\{z\},\{X\})$ but no
$\tau_w$-neighbourhood of $L$ in $\Li(X)$ is contained in $\U(\{z\},\{X\})$ thus the identity map from
$(\Li(X),\tau_w)$ to $(\Li(X),\tau_s)$ is not continuous at $L$. Indeed, if $\U(\mset,\Phi)$ is any basic
$\tau_w$-neighbourhood of $L$ then $L\cup \overline{\{z\}}\in \U(\mset,\Phi)\cap \Li(X)$ but $L\cup
\overline{\{z\}}\notin \U(\{z\},\{X\})$.

\end{proof}

\begin{cor}

   The restrictions of $\tau_w$ and $\tau_s$ to $\ML(X)$ coincide.

\end{cor}

A point $y$ of a topological space $Y$ is called separated in $Y$ if, for every $z\in Y\setminus \overline{\{y\}}$,
$y$ and $z$ have disjoint neighbourhoods; equivalently, $\overline{\{y\}}$ is a maximal limit set(see
\cite[D\`{e}finition 16]{D}). It is proved in \cite[Th\`eor\'eme 19]{D} that if $Y$ is a second countable locally
compact Baire space then the subset of all separated points in $Y$ is a dense $G_{\delta}$. For any topological
space $X$ the density of the set of all separated points in $(\Li(X),\tau_w)$ is an immediate corollary of the next
result.

\begin{thm} \label{T:separated}

   Let $X$ be a topological space. An element $A$ of $\Li(X)$ is separated in $(\Li(X),\tau_w)$ if and only if $A$
   is a maximal limit set.

\end{thm}

\begin{proof}

If $A\in \Li(X)$ is not maximal then there is $A_1\in \Li(X)$ such that $A_1\supsetneqq A$. Then $A_1$ does not
belong to the $\tau_w$-closure of $\{A\}$ in $\Li(X)$. However, $A$ is in the $\tau_w$-closure of $\{A_1\}$ in
$\Li(X)$ hence $A$ and $A_1$ cannot be separated by disjoint $\tau_w$-open sets.

Suppose now that $A$ is a maximal limit set and $A_1\in \Li(X)$ does not belong to the $\tau_w$-closure of $\{A\}$
that is, $A_1$ is not included in $A$. Then $A\cup A_1\in \F(X)\setminus \Li(X)$. By \cite[Lemme 9]{D} there is a
finite family $\Phi$ of open subsets of $X$ such that each of them has a nonempty intersection with $A\cup A_1$ but
the intersection of all the sets in $\Phi$ is void. Let $\Psi$ be the subfamily of $\Phi$ consisting of those sets
that have a nonempty intersection with $A$. Since $A_1\in \Li(X)$ we must have, by the above quoted lemma of
Dixmier, $\Psi\neq \mset$. Similarly, $\Psi_1 := \Phi\setminus \Psi$ is not empty since $A\in \Li(X)$. Now,
$\U(\mset,\Psi)\cap \Li(X)$ is a $\tau_w$-neighbourhood of $A$ in $\Li(X)$ and $\U(\mset,\Psi_1)\cap \Li(X)$ is a
$\tau_w$-neighbourhood of $A_1$ in $\Li(X)$. We have
\[
 \U(\mset,\Psi)\cap \U(\mset,\Psi_1)\cap \Li(X) = \mset
\]
hence $A$ and $A_1$ can be separated by disjoint $\tau_w$-open sets. Indeed, if the above equality does not hold and
$B\in \U(\mset,\Psi)\cap \U(\mset,\Psi_1)\cap \Li(X)$ then
\[
 \cap\{V \mid V\in \Phi\} = (\cap \{V \mid V\in \Psi\})\bigcap (\cap \{V \mid V\in \Psi_1\})\neq \mset
\]
since $B\in \Li(X)$, a contradiction.

\end{proof}

The following two propositions were stated and proved in \cite{A} in the language of $C^*$-algebras. We only had to
rewrite the proofs to be fit in a more general situation.

\begin{prop}[{\cite[Proposition 4.9]{A}}] \label{P:BaireML}

   $\ML(X)$ is a Baire space if $X$ is a Baire space.

\end{prop}

\begin{proof}

   Let $\{\V_n\}$ be a sequence of $\tau_w$-open subsets of $\Li(X)$ such that every $\U_n := \V_n\cap \ML(X)$ is dense
   in $\ML(X)$. Since $\ML(X)$ is $\tau_w$-dense in $\Li(X)$ we get that each $\V_n$ is $\tau_w$-dense in
   $\Li(X)$. By Proposition \ref{P:Baire}, $\cap \V_n$ is $\tau_w$-dense in $\Li(X)$. Let now $\U$ be an open set in
   $\ML(X)$. Then $\U = \V\cap \ML(X)$, $\V$ being a $\tau_w$-open set in
   $\Li(X)$. There exist $B\in (\cap \V_n)\cap \V$ and $B_1\in \ML(X)$ with $B\subset B_1$. Since $B_1$ belongs to
   any $\tau_w$-open set of $\Li(X)$ to which $B$ belongs, we have
   \[
    B_1\in (\cap \V_n)\cap \V\cap \ML(X) = (\cap \U_n)\cap \U.
   \]
   Hence $\cap \U_n$ is dense in $\ML(X)$.

\end{proof}

\begin{prop}[{\cite[Corollary 4.6]{A}}] \label{P:seppts}

   If $X$ is a second countable locally compact Baire space then the family $\{\overline{\{x\}} \mid x \mbox{\quad \emph{is
   separated in}} \quad  X\}$ is dense in $\ML(X)$.

\end{prop}

\begin{proof}

   As mentioned before, \cite[Th\`eor\'eme 19]{D} asserts that the set
   \[
    T := \{x\in X \mid x \mbox{\quad is separated in} \quad X\}
   \]
   is dense in $X$. Since $\eta_X(X)$ is $\tau_w$-dense in $\Li(X)$ and $\eta_X$ is $\tau_w$-continuous we can infer
   that $\eta_X(T)$ is $\tau_w$-dense in $\Li(X)$. In particular, $\eta_X(T) = \eta_X(X)\cap \ML(X)$ is dense in
   $\ML(X)$.

\end{proof}

We shall have more to say about the set considered in the statement of Proposition \ref{P:seppts} in Corollary
\ref{C:gdelta}

Theorems \ref{T:cont} and \ref{T:separated} give us some information about the way $\ML(X)$ is imbedded in $\Li(X)$.
Theorem \ref{T:delta} will show us another aspect of this imbedding when the space is second countable. First we
need two lemmas.

\begin{lem} \label{L:open}

   Let $Y$ be a compact space, $M\subset Y\times Y$ and
   \[
    S(M) := \{y\in Y \mid \{y\}\times Y\subset M\}.
   \]
   If $M$ is open then $S(M)$ is open and if $M$ is a $G_{\delta}$ set then $S(A)$ is a $G_{\delta}$ set too.

\end{lem}

\begin{proof}

   Suppose $M$ is an open set. If $y\in S(A)$ then, by using the compactness of $Y$, we can infer that there are
   open subsets $\{U_i\}_{i=1}^n$ and $\{V_i\}_{i=1}^n$ of $Y$ such that
   \[
    \{y\}\times Y\subset \cup_{i=1}^n (U_i\times V_i)\subset M.
   \]
   Then
   \[
    y\in \cap_{i=1}^n U_i\subset S(M)
   \]
   and $\cap_{i=1}^n U_i$ is open.

   Suppose now $M$ is a $G_{\delta}$ set, $M = \cap_1^{\infty} M_n$ with each $M_n$ open in $Y\times Y$. Since $S(M) =
   \cap_1^{\infty} S(M_n)$ and $S(M_n)$ is open by the first part of the proof, the conclusion obtains.

\end{proof}

\begin{lem} \label{L:closed}

   Let $X$ be a locally compact space. Then
   \[
    \E := \{(A,B)\in \Li(X)\times \Li(X) \mid A\subset B\}
   \]
   is $(\tau_s\times \tau_s)$-closed in $\Li(X)\times \Li(X)$.

\end{lem}

\begin{proof}

   Let $\{(A_{\iota},B_{\iota})\}$ be a net in $\E$ that $(\tau_s\times \tau_s)$-converges to $(A,B)$.
   Given $x\in A$ there exists, by Proposition \ref{P:conv1}, a subnet $\{A_{\iota_{\kappa}}\}$ of $\{A_{\iota}\}$ and
   points $x_{\iota_{\kappa}}\in A_{\iota_{\kappa}}\subset B_{\iota_{\kappa}}$ such that $\{x_{\iota_{\kappa}}\}$
   converges to $x$. Again by Proposition \ref{P:conv1}, $x\in B$ and we have shown $A\subset B$ that is $(A,B)\in
   \E$.
\end{proof}

\begin{thm} \label{T:delta}
   If $X$ is a second countable locally compact space then $\ML(X)$ is a $G_{\delta}$ subset of $(\Li(X),\tau_s)$.
\end{thm}

\begin{proof}

   Set
   \[
    \mathcal{D} := \{(A,A) \mid A\in \Li(X)\}, \quad \E := \{(A,B)\in \Li(X)\times \Li(X) \mid A\subset B\},
   \]
   and
   \[
    \mathcal{T} := \Li(X)\times \Li(X)\setminus (\E\setminus \mathcal{D}).
   \]
   Then for $A\in \Li(X)$ we have $A\in \ML(X)$ if and only if $\{A\}\times \Li(X)\subset \mathcal{T}$.

\end{proof}

\begin{rem}

   If $X$ is a second countable locally compact space then $\ML(X)$ is a Baire space since it is a $G_{\delta}$
   subset of the compact metrizable space $(\Li(X),\tau_s)$.

\end{rem}

\begin{cor} \label{C:gdelta}

   If $X$ is a second countable locally compact space in which every closed subset is a Baire space then
   $\{\overline{\{x\}} \mid x \quad \mbox{\emph{is separated in}} \quad X\}$ is a $G_{\delta}$ subset of $\Li(X)$; it is also a
   dense subset of $\ML(X)$.

\end{cor}

\begin{proof}

By \cite[Theorem 7]{D}, $\eta_X(X)$ is a $G_{\delta}$ subset of $\F(X)$ hence it is a $G_{\delta}$ subset of
$\Li(X)$. Then $\{\overline{\{x\}} \mid x \quad \mbox{\emph{is separated in}} \quad X\} = \eta_X(X)\cap \ML(X)$ is a
dense $G_{\delta}$ subset of $\ML(X)$ by Proposition \ref{P:seppts} and a $G_{\delta}$ subset of $\Li(X)$ by Theorem
\ref{T:delta}.

\end{proof}

\begin{rem}

The primitive ideal space of a separable $C^*$-algebra with its hull-kernel topology satisfies the hypothesis of
Corollary \ref{C:gdelta}.

\end{rem}
 ----------------------------------------------------------------
\bibliographystyle{amsplain}
\bibliography{}

\end{document}